\documentclass{article}
\usepackage{amsfonts}
\usepackage{amsmath}
\usepackage{amsthm}
\usepackage{amssymb}
\usepackage[all]{xy}
\usepackage{url}

\def\@begintheorem#1#2{\par\bgroup{\sc #1\ #2. }\it\ignorespaces}
\def\@opargbegintheorem#1#2#3{\par\bgroup{\sc #1\ #2\ (#3). } \it\ignorespaces}
\def\@endtheorem{\egroup}
\newtheorem{theorem}{Theorem}[section]
\newtheorem{corollary}[theorem]{Corollary}

\newtheorem{problem}[theorem]{Problem}
\newtheorem{example}[theorem]{Example}
\newtheorem{remark}[theorem]{Remark}

\newtheorem{definition}[theorem]{Definition}

\newtheorem{question}[theorem]{Question}

\def\dist{\operatorname{dist}}

\begin{document}

\title{Polyhedral graph abstractions and an approach to the Linear Hirsch Conjecture}
\author{Edward D. Kim\thanks{Supported by Vidi grant 639.032.917 from the Netherlands Organization for Scientific Research (NWO)}\\{\sl Technische Universiteit Delft}}
\date{}

\maketitle

\begin{abstract}
We introduce a new combinatorial abstraction for the graphs of polyhedra. The new abstraction is a flexible framework defined by combinatorial properties, with each collection of properties taken providing a variant for studying the diameters of polyhedral graphs. One particular variant has a diameter which satisfies the best known upper bound on the diameters of polyhedra. Another variant has superlinear asymptotic diameter, and together with some combinatorial operations, gives a concrete approach for disproving the Linear Hirsch Conjecture.
\end{abstract}

\noindent{\small {\bf MSC:} 05C12, 52B05, 52B40, 90C05\\
{\bf Keywords:} convex polytopes, Hirsch Conjecture, linear programming}

\section{Introduction}\label{section:introduction}

Studying the diameters of the graphs of polytopes and polyhedra has received a lot of attention~\cite{kalaiblog:polymath3part6} due to Santos' recent counter-example~\cite{Santos:CounterexampleHirsch} to the Hirsch Conjecture. The Hirsch Conjecture asserts that the diameter of any $d$-dimensional polytope with $n$ facets is never greater than $n-d$. Since this conjecture is now known to be false, the question of the Polynomial Hirsch Conjecture, which asserts that the diameter of any polytope with $n$ facets is polynomial in $n$, is relevant. (The dimension $d$ is not in the statement of the Polynomial Hirsch Conjecture since $n > d$.) The first step in this line of investigation is to settle the Linear Hirsch Conjecture, which asserts that the diameter is linear in the number $n$ of facets, independent of the dimension $d$.

The original Hirsch Conjecture was stated for the graphs of polyhedra. Since Klee and Walkup~\cite{Klee:d-step} showed that the Hirsch Conjecture is false for unbounded polyhedra, the Hirsch Conjecture (and the Linear and Polynomial Hirsch Conjectures, which are both still open) is usually stated for bounded polytopes. However, we should note that the Linear Hirsch or Polynomial Hirsch Conjectures would still be interesting for unbounded polyhedra.

These conjectures on the diameters of polytopes and polyhedra are interesting because of their relation to the efficiency of the simplex algorithm for linear programming. In particular, the diameter of the feasibility polyhedron is a lower bound on the number of pivot steps needed for the simplex algorithm. Thus, if the Polynomial Hirsch Conjecture is false, then no pivot rule for the simplex algorithm runs in polynomial time. For more information on the Hirsch Conjecture and its relationship to the behavior of the simplex method, see the survey~\cite{Klee:The-d-step-conjecture} or the recent survey~\cite{KimSantos:HirschSurvey} written jointly with Santos. Our terminology on polytopes follows the language in~\cite{Ziegler:Lectures}.

The study of diameters of convex polyhedra via combinatorial abstractions of polyhedral graphs was considered by Adler, Dantzig, Murty, and Saigal~\cite{Adler:AbstractPolytopesThesis, AdlerDantzigMurty:AbstractPolytopes, Adler:MaxDiamAbsPoly, AdlerSaigal:LongPathsAbstract, Murty:GraphAbstract}, Kalai~\cite{Kalai:DiameterHeight}, and Eisenbrand, H\"ahnle, Razborov, and Rothvo\ss{}~\cite{Eisenbrand:Diameter-of-Polyhedra, Hahnle:Diplomathesis}. Adler et al.{} introduced the first formally-defined combinatorial abstraction of the graphs of polytopes satisfying a collection of axioms. In~\cite{Kalai:DiameterHeight}, Kalai showed that a more general family of objects obtained by removing one of these axioms still satisfies the quasi-polynomial upper bound for the diameters of convex polyhedra proved in~\cite{Kalai:Quasi-polynomial}. A further generalization of polyhedral graphs was studied in~\cite{Eisenbrand:Diameter-of-Polyhedra} by dropping an additional axiom. Eisenbrand et al.{} prove that even for this more general class of objects, the quasi-polynomial diameter upper bound of Kalai and Kleitman in~\cite{Kalai:Quasi-polynomial} holds. On the other hand, they give a construction to prove that the diameters of some objects in this more general class are superlinear. One can say that this superlinear lower bound in~\cite{Eisenbrand:Diameter-of-Polyhedra} is evidence against the Linear Hirsch Conjecture.

In this paper, we introduce \emph{subset partition graphs} (defined in Section~\ref{section:subset-partition-graphs}), a new family of combinatorial abstractions of the graphs of polyhedra. The new abstraction is a flexible framework inspired by the combinatorial properties found in previous abstractions, and provides many variants for abstractions of polyhedra. Our main theoretical result is the construction of a family of subset partition graphs whose diameter is superpolynomial (see Theorem~\ref{theorem:general-lower}), which can be considered evidence against the Polynomial and Linear Hirsch Conjectures. We also prove a quadratic lower bound on the diameters of a special subclass of subset partition graphs, which provides further evidence against the Linear Hirsch Conjecture. Moreover, we present a strategy to disprove the Linear Hirsch Conjecture via combinatorial operations on subset partition graphs.

\vskip12pt

\noindent{\bf Outline of this paper:} In Section~\ref{section:previous-abstractions-main}, we formally define the previous abstractions and survey the known upper and lower bounds. Section~\ref{section:base-abstractions-clf} introduces the abstraction presented in~\cite{Eisenbrand:Diameter-of-Polyhedra} and Section~\ref{section:previous-abstractions} discusses combinatorial properties defining special cases. Motivated by this discussion, in Section~\ref{section:subset-partition-graphs} we define our new combinatorial abstraction, the subset partition graphs. In Section~\ref{section:bounds} we prove upper and lower bounds on the diameters of subset partition graphs that satisfy particular sets of properties. We give some final remarks in Section~\ref{section:final-remarks}, and present a strategy for disproving the Linear Hirsch Conjecture.

\section{Previous abstractions}\label{section:previous-abstractions-main}

Here we describe relevant previous combinatorial abstractions in the literature\footnote{Very recently, new abstractions unrelated to Section~\ref{section:subset-partition-graphs} introduced in~\cite{kalaiblog:polymath3part6} have formed the discussion of a web-based discussion: See \url{http://gilkalai.wordpress.com/2010/09/29/polymath-3-polynomial-hirsch-conjecture/}, \url{http://gilkalai.wordpress.com/2010/10/03/polymath-3-the-polynomial-hirsch-conjecture-2/}, \url{http://gilkalai.wordpress.com/2010/10/10/polymath3-polynomial-hirsch-conjecture-3/}, \url{http://gilkalai.wordpress.com/2010/10/21/polymath3-polynomial-hirsch-conjecture-4/},   \url{http://gilkalai.wordpress.com/2010/11/28/polynomial-hirsch-conjecture-5-abstractions-and-counterexamples/}, and \url{http://gilkalai.wordpress.com/2011/04/13/polymath3-phc6-the-polynomial-hirsch-conjecture-a-topological-approach/}.}. In all cases, the object of study is an abstract generalization of simple polyhedra. We say that a $d$-dimensional polyhedron $P$ is \emph{simple} if each of its vertices is contained in exactly $d$ of the $n$ facets of $P$. For the study of diameters of polyhedra, it is enough to consider simple polyhedra, since the largest diameter of a $d$-polyhedron with $n$ facets is found among the simple $d$-polyhedra with $n$ facets~\cite{KimSantos:HirschSurvey}.

Let $H(n,d)$ denote the maximum diameter of $d$-dimensional polytopes with $n$ facets, and let $H_u(n,d)$ denote the maximum diameter of $d$-dimensional polyhedra with $n$ facets. Since polytopes are polyhedra, we clearly have $H(n,d) \leq H_u(n,d)$. In~\cite{Kalai:Quasi-polynomial}, Kalai and Kleitman proved that $H_u(n,d) \leq n^{1+ \log_2 d}$.

\subsection{Base abstractions and connected layer families}\label{section:base-abstractions-clf}

We now describe an abstraction of Eisenbrand et al.{} introduced in~\cite{Eisenbrand:Diameter-of-Polyhedra}. Fix a finite set $S$ of cardinality $n$, called the \emph{symbol set}. (Each $s \in S$ is called a \emph{symbol}.) Let $\mathcal{A} \subseteq \binom{S}{d}$, where $\binom{S}{d}$ is the set of all $d$-element subsets of $S$. We consider connected graphs of the form $G = (\mathcal{A}, E)$ with  vertex set $\mathcal{A}$ and edge set $E$ satisfying:
\begin{itemize}
\item for each $A, A' \in \mathcal{A}$, there is a path from $A$ to $A'$ in the graph $G$ using only vertices that contain $A \cap A'$.
\end{itemize}
If this occurs, we say that $G$ is a $d$-dimensional \emph{base abstraction} of $\mathcal{A}$ on the symbol set $S$. The \emph{diameter} of the base abstraction is the diameter of the graph $G$.

Note that the graphs of simple $d$-dimensional polyhedra with $n$ facets are base abstractions. Indeed, each of the $n$ facets of $P$ is associated with a symbol $s$ in $S$. Since our polyhedron $P$ is simple, each vertex of $P$ is incident to exactly $d$ facets, and so it is associated with the $d$-element subset of $S$ consisting of the corresponding symbols. The graph $G$ used in the base abstraction ``is'' the graph of the polyhedron. The defining condition is satisfied since, for every pair of vertices $y$ and $z$ on a polyhedron $P$, there is a path from $y$ to $z$ on the smallest face of $P$ containing both $y$ and $z$.

Since the graphs of polyhedra are base abstractions, we clearly have $H_u(n,d) \leq B(n,d)$, where $B(n,d)$ denotes the maximum diameter among $d$-dimensional base abstractions on a symbol set of size $n$. In~\cite{Eisenbrand:Diameter-of-Polyhedra}, Eisenbrand et al.{} prove that the Kalai-Kleitman bound of $n^{1+\log_2 d}$ is also an upper bound for $B(n,d)$. They also show that $B(n,\frac{n}{4})$ is in $\Omega(n^2 / \log n)$, i.{}e.{}, the diameters of base abstractions obey a quadratic lower bound (up to a logarithmic factor).

Their upper and lower bounds for $B(n,d)$ were proved by analyzing the diameters of a related combinatorial object. A $d$-dimensional \emph{connected layer family} of $\mathcal{A} \subseteq \binom{S}{d}$ on a set of $n = |S|$ symbols is a family $\mathcal{V} = \{ \mathcal{V}_0, \ldots, \mathcal{V}_t \}$ of non-empty sets such that:
\begin{itemize}
\item {\bf partition property:} $\mathcal{A} = \mathcal{V}_0 \cup \cdots \cup \mathcal{V}_t$,
\item {\bf disjointness property:} $\mathcal{V}_i \cap \mathcal{V}_j = \emptyset$ if $i \not= j$,
\item {\bf connectivity property:} for all $i < j < k$ and $A \in \mathcal{V}_i$, $A' \in \mathcal{V}_k$, there is an $A'' \in \mathcal{V}_j$ such that $A \cap A' \subseteq A''$.
\end{itemize}
Each individual $\mathcal{V}_i$ is called a \emph{layer}. The \emph{diameter} of the connected layer family $\mathcal{V} = \{ \mathcal{V}_0, \ldots, \mathcal{V}_t \}$ is $t$. Recall that $B(n,d)$ denotes the maximal diameter of a $d$-dimensional base abstraction on a symbol set of size $n$. We use  $C(n,d)$ to denote the maximal diameter of a $d$-dimensional connected layer family on a symbol set of size $n$. 

In~\cite{Eisenbrand:Diameter-of-Polyhedra}, Eisenbrand et al.{} prove $B(n,d) = C(n,d)$. The proof that $B(n,d) \geq C(n,d)$ follows from the fact that a base abstraction is obtained from a connected layer family by connecting $d$-sets $A \in \mathcal{V}_i$ and $A' \in \mathcal{V}_j$ with an edge if $|i-j| \leq 1$. The proof of $B(n,d) \leq C(n,d)$ follows from the fact that a connected layer family is obtained from a base abstraction by the following layering process: let $G = (\mathcal{A}, E)$ be a $d$-dimensional base abstraction, and fix a particular $d$-subset $Z \in \mathcal{A}$. Then, let $\mathcal{V}_i := \{ A \in \mathcal{A} : \dist_G(A,Z) = i\}$. Note that $\mathcal{V} = \{ \mathcal{V}_0, \ldots, \mathcal{V}_t \}$ is a $d$-dimensional connected layer family with diameter $t$ since the face path property immediately implies that the collection $\mathcal{V}$ satisfies the connectivity property. See Lemma~3.4.2 in~\cite{Hahnle:Diplomathesis} for a detailed proof. The bound of $B(n,d) \leq C(n,d)$ is obtained by choosing a base abstraction $G=(\mathcal{A},E)$ whose diameter is $B(n,d)$ and picking a pair $(Z,Z')$ of $d$-sets in $\mathcal{A}$ at distance $B(n,d)$.

\subsection{Previous abstractions satisfying additional properties}\label{section:previous-abstractions}

Our new abstraction defined in Section~\ref{section:subset-partition-graphs} is motivated by following the same layering process just described, but starting with special cases of base abstractions that were studied in~\cite{AdlerDantzigMurty:AbstractPolytopes} and~\cite{Kalai:DiameterHeight} which satisfied additional combinatorial properties. Let $G = (\mathcal{A}, E)$ be a $d$-dimensional base abstraction. If the condition ``$(A,A')$ is an edge in $E$ if and only if $|A \cap A'| = d-1$'' holds, then we say that $G$ is an \emph{ultraconnected set system}. This condition is called \emph{ultraconnectedness}. Note that ultraconnectedness holds for the graphs of polyhedra: if two vertices $y$ and $z$ of a simple $d$-polyhedron $P$ share all but one facet in common, then they are neighbors in the graph of $P$. In~\cite{Kalai:DiameterHeight}, Kalai proved that ultraconnected set systems satisfy the diameter bound of~\cite{Kalai:Quasi-polynomial}.

If the layering process described earlier is applied to a base abstraction satisfying the ultraconnectedness property, then the resulting collection $\mathcal{V} = \{ \mathcal{V}_0, \ldots, \mathcal{V}_t \}$ of layers satisfies the following adjacency property: if $A, A' \in \mathcal{A}$ and $|A \cap A'| = d-1$, then $A$ and $A'$ are in the same or adjacent layers, i.e., if $A \in \mathcal{V}_i$ and $A' \in \mathcal{V}_j$ then $|i-j| \in \{0,1\}$.

Adler, Dantzig, and Murty considered an abstraction where, in addition to ultraconnectedness, the following \emph{polytopal endpoint-count condition} must hold: if $F \in \binom{S}{d-1}$, then $\{A \in \mathcal{A} : F \subset A\}$ has cardinality either $0$ or $2$. An ultraconnected set system satisfying this additional condition is called an \emph{abstract polytope}. For polytopes, the polytopal endpoint-count condition translates into the fact that every $1$-face of a polytope is incident to exactly two $0$-faces. This condition fails for polyhedra because of $1$-faces containing only one vertex, so we consider a \emph{polyhedral endpoint-count condition}, where we allow $|\{A \in \mathcal{A} : F \subset A\}|$ to be $1$ as well. In~\cite{Adler:MaxDiamAbsPoly}, Adler and Dantzig prove that $d$-dimensional abstract polytopes with $n$ symbols satisfy the Hirsch bound if $n-d \leq 5$, the abstract analogue of~\cite{Klee:d-step}.

\section{Subset partition graphs}\label{section:subset-partition-graphs}

Now that we have seen how properties defining special classes of base abstractions imply certain structural properties on connected layer families obtained by the layering process, we are ready to define subset partition graphs, which are a generalization of base abstractions and connected layer families. As before, we have a set $S$ called the symbol set, with each $s \in S$ being a symbol.

\begin{definition}
Fix a finite set $S$ of cardinality $n$ and a set $\mathcal{A} \subseteq \binom{S}{d}$ of subsets. Let $G = (\mathcal{V},E)$ be a connected graph with vertex set $\mathcal{V} = \{\mathcal{V}_0,\ldots,\mathcal{V}_t\}$. If $\mathcal{V}$ is a partition of $\mathcal{A}$ in the sense that:
\begin{enumerate}
\item $\mathcal{A} = \mathcal{V}_0 \cup \cdots \cup \mathcal{V}_t$,
\item $\mathcal{V}_i \cap \mathcal{V}_j = \emptyset$ if $i \not= j$, and
\item $\mathcal{V}_i \not= \emptyset$ for all $i$,
\end{enumerate}
then we say that $G$ is a $d$-dimensional \emph{subset partition graph} of $\mathcal{A}$ on the symbol set $S$.
\end{definition}
Note that $d$-dimensional subset partition graphs on a symbol set $S$ of size $n$ are combinatorial abstractions of simple $d$-dimensional polyhedra with $n$ facets: each of the $n$ facets of a $d$-dimensional polyhedron $P$ corresponds to a symbol $s \in S$, and a vertex of $P$ corresponds to a $d$-set $A \in \mathcal{A}$ given by the incident facets. 

As defined, the only condition on the edge set $E$ is that the graph $G$ is connected, and thus subset partition graphs do not yet give an interesting combinatorial abstraction of the graphs of polytopes and polyhedra. For this, one should require one or more of the combinatorial properties identified below, which are conditions on the set $\mathcal{A}$ of subsets or on the edge set $E$ of the graph $G$. Before identifying the properties, we need to define the following operation on subset partition graphs:
\begin{definition}[Restriction]
Let $G = (\mathcal{V},E)$ be a subset partition graph of $\mathcal{A}$ on the symbol set $S$, and let $F \subseteq S$ be a collection of symbols. We define a new subset partition graph $G_F = (\mathcal{V}_F, E_F)$ of $\mathcal{A}_F$ on the symbol set $S' := S$. 

We define $\mathcal{A}_F := \{A \in \mathcal{A} : F \subseteq A\}$. That is to say, $\mathcal{A}_F$ is obtained by deleting from $\mathcal{A}$ (and the containing $\mathcal{V}_i$) any $d$-set $A$ which does not contain $F$. This deletion from the vertices in $\mathcal{V}$ may make some of them empty. The vertex set $\mathcal{V}_F$ consists of those vertices in $\mathcal{V}$ which are still non-empty, and two vertices in $\mathcal{V}_F$ are connected by an edge in $E_F$ exactly when the associated vertices were connected in $E$. The subset partition graph $G_F$ is called the \emph{restriction} of $G$ with respect to $F$.
\end{definition}

Based on the discussion in Sections~\ref{section:base-abstractions-clf} and~\ref{section:previous-abstractions}, and together with the definition of restriction, we identify the main properties that should be considered for subset partition graphs:
\begin{itemize}
\item {\bf dimension reduction:} if $F \subseteq S$ such that $|F| \leq d$, then (the underlying graph of) the restriction $G_F$ is a connected graph.
\item {\bf adjacency:} if $A, A' \in \mathcal{A}$ and $|A \cap A'| = d-1$, then $A$ and $A'$ are in the same or adjacent vertices of $G$.
\item {\bf strong adjacency:} adjacency holds and, if two vertices $\mathcal{V}_i$ and $\mathcal{V}_j$ are adjacent in $G$ then there are $d$-sets $A \in \mathcal{V}$ and $A' \in \mathcal{V}_j$ such that $|A \cap A'| = d-1$.
\item {\bf endpoint-count:} if $F \in \binom{S}{d-1}$, then $|\{A \in \mathcal{A} : F \subset A\}| \leq 2$.
\end{itemize}
Note that the graphs of polytopes satisfy dimension reduction since the graph of a face is connected. The notion of strong adjacency was proposed by H\"{a}hnle in~\cite{Hahnle:SPGs}.
\begin{example}\label{example:non-polytopal-spg}
Figure~\ref{figure:example-spg} illustrates a $3$-dimensional subset partition graph on a symbol set $S$ with $n=|S|=6$ which satisfies the dimension reduction, strong adjacency, and endpoint-count properties. The graph of $G$ has six vertices and $|\mathcal{A}| = 2 \times 2 + 4 \times 1$.
\begin{figure}[hbt]
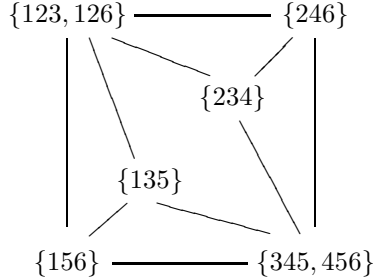

\[ 
\xy 
(0,0)*{\{123,126\}}; 
(33,0)*{\{246\}}; 
(0,-33)*{\{156\}};
(33,-33)*{\{345,456\}};
(22,-11)*{\{234\}};
(11,-22)*{\{135\}};
(9,0); (27,0) **\dir{-}; 
(6,-33); (24,-33) **\dir{-}; 
(0,-3); (0,-28) **\dir{-}; 
(33,-3); (33,-28) **\dir{-}; 
(30,-3); (25,-8) **\dir{-};
(3,-29.3); (8,-25) **\dir{-};
(3,-3); (9, -19) **\dir{-};
(6,-3); (19, -8) **\dir{-};
(27,-29); (12, -25) **\dir{-};
(31,-29); (23, -14) **\dir{-};
\endxy 
\]
\caption{A subset partition graph $G$ with $d=3$ and $n=6$.}\label{figure:example-spg}
\end{figure}
\end{example}
We define these (and the following) properties for subset partition graphs because we want a flexible framework where we consider certain collections of properties at a time. By considering certain properties ``on'' and other properties ``off'' we hope to understand which properties are crucial in proving diameter bounds.
For instance, the class of subset partition graphs with the dimension reduction property together with the additional condition that the underlying graph of $G$ is a path is exactly the class of connected layer families. Subset partition graphs interpolate between connected layer families and polyhedral graphs by the selection of properties.
(Note that the properties presented are not necessarily independent. For instance, endpoint-count together with dimension reduction implies adjacency. However, we explicitly list the adjacency property as we will consider subset partition graphs satisfying adjacency but not satisfying dimension reduction.)

%

In addition to the main properties described above, there are other combinatorial properties of polytopes which translate into natural properties to consider for subset partition graphs, for example:
\begin{itemize}
\item {\bf $d$-connectedness:} the graph $G$ is $d$-connected.
\item {\bf $d$-regularity:} the graph $G$ is $d$-regular.
\item {\bf $d$-neighbors:} for every $A \in \mathcal{A}$, $|\{A' \in \mathcal{A} \setminus \{A\} : |A \cap A'| = d-1\} | = d$.
\item {\bf one-subset:} $|\mathcal{V}_i| = 1$ for each $i = 0, \ldots, t$.
\end{itemize}
The $d$-connectedness property for subset partition graphs is desirable due to Balinski's Theorem, which says that the graph of a $d$-dimensional polytope is $d$-connected~\cite{Balinski:On-the-graph-structure}. The $d$-regularity and $d$-neighbors property hold for the graphs of simple $d$-polytopes: at each vertex $v$ of a simple $d$-polytope $P$, there are $d$ edges emanating from $v$, and in a bounded polytope, each of these edges leads to another vertex of $P$. These properties do not hold for unbounded polyhedra. The one-subset property, which says that each vertex should have a single $d$-subset, holds for the graphs of polytopes: each vertex in the subset partition graph should contain the $d$-set of incident facets.

There are two easy operations one can perform on a subset partition graph $G = (\mathcal{V},E)$. Let $\mathcal{V}_i$ and $\mathcal{V}_j$ be two vertices in $\mathcal{V}$. Then:
\begin{enumerate}
\item {\bf Contraction:} If $\mathcal{V}_i$ and $\mathcal{V}_j$ are connected by an edge in $E$, contraction on the edge produces a new subset partition graph with one less vertex: the two vertices $\mathcal{V}_i$ and $\mathcal{V}_j$ are replaced with a new vertex which contains all of the $d$-sets which were in $\mathcal{V}_i$ and $\mathcal{V}_j$.
\item {\bf Edge addition:} If $\mathcal{V}_i$ and $\mathcal{V}_j$ were not connected by an edge in $E$, edge addition makes the two vertices adjacent. The resulting subset partition graph has one more edge than the original subset partition graph $G$ does.
\end{enumerate}
\begin{example}
The subset partition graph described in Example~\ref{example:non-polytopal-spg} is obtained from the natural subset partition graph for a $3$-cube after two contractions.
\end{example}
We remark that there is a clear analogue for contraction in the theory of connected layer families. We also note the following simple but potentially powerful effect of these operations, which will be important in Section~\ref{section:final-remarks}:
\begin{remark}\label{remark:effect-of-operations}
We wish to note what happens to the dimension reduction, adjacency, and endpoint-count properties for subset partition graphs under the operations of contraction and edge addition:
\begin{enumerate}
\item All three of these properties are preserved under both operations.
\item After a sufficient number of contractions and edge additions, the resulting graph will be a complete graph, and thus the dimension reduction and adjacency conditions will hold.
\end{enumerate}
\end{remark}

\section{Upper and lower bounds}\label{section:bounds}

In this section, we prove upper and lower bounds on the diameters of subset partition graphs satisfying various combinations of the main properties. First, note the following easy bound:
\begin{remark}\label{remark:cfl-one-subset-bound}
Recall that if we consider subset partition graphs with the dimension reduction property, together with the condition that the graph $G$ is a path, then this is exactly the same as studying connected layer families. If we also add the one-subset property, then the Hirsch bound of $n-d$ holds: up to permutation of symbols, the unique longest path is $\{1, \ldots, d\}, \{2, \ldots, d+1\}, \ldots, \{n-d+1, \ldots, n\}$.
\end{remark}

\subsection{Upper bound}

Here we prove that the diameters of subset partition graphs satisfying the dimension reduction property obey the Kalai-Kleitman quasi-polynomial upper bound of $n^{1+\log_2 d}$.

\begin{theorem}
Let $J(n,d)$ denote the maximal diameter among $d$-dimensional subset partition graphs on $n$ symbols which satisfy dimension reduction. Then, $J(n,d) \leq n^{1+\log_2 d}$.
\end{theorem}
\begin{proof}
Let $G = (\mathcal{V},E)$ be an arbitrary $d$-dimensional subset partition graph of $\mathcal{A}$ on $n$ symbols of maximal diameter which satisfies dimension reduction. Let $A_1$ and $A_2$ be two arbitrary subsets belonging to $\mathcal{A}$ with distance $J(n,d)$. We will construct a $d$-dimensional connected layer family with $n$ symbols whose diameter is $J(n,d)$ which proves it is bounded above by the maximal diameter $C(n,d)$ of connected layer families.

The connected layer family is obtained from the layering process. We define $\mathcal{F}_i := \{ A \in \mathcal{A} : \dist_G(A,A_1) = i\}$. We claim that $\mathcal{F} = \{ \mathcal{F}_0, \ldots, \mathcal{F}_{J(n,d)} \}$ is a $d$-dimensional connected layer family. To prove this, let $i < j < k$ be arbitrary and $A \in \mathcal{F}_i$, $A' \in \mathcal{F}_k$. We prove that there is an $A'' \in \mathcal{F}_j$ such that $F \subseteq A''$, where $F := A \cap A'$. For a contradiction, suppose that there is no $A'' \in \mathcal{F}_j$ containing $F$.

Since $d$-sets in adjacent vertices of $G$ are in the same or adjacent layer of $ \{ \mathcal{F}_0, \ldots, \mathcal{F}_{J(n,d)}\}$, removing all vertices (and incident edges) containing a $d$-set in $\mathcal{F}_j$ in the graph $G$ disconnects the $d$-sets in $\mathcal{F}_i$ from the $d$-sets in $\mathcal{F}_k$. Thus the vertices containing $A$ and $A'$ lie in distinct connected components of $G_F$, contradicting the dimension reduction property. Therefore $J(n,d) \leq C(n,d) \leq n^{1+\log_2 d}$.
\end{proof}

\subsection{Lower bounds}

In this section, we prove lower bounds on the diameters of subset partition graphs satisfying the adjacency and endpoint-count conditions. First, we prove a general lower bound. Then we prove a lower bound for a special subclass of subset partition graphs satisfying natural combinatorial properties coming from an interesting class of polytopes.

\begin{remark}\label{remark:spg-lower-bound-from-clf}
The construction of Eisenbrand et al.{} in~\cite{Eisenbrand:Diameter-of-Polyhedra} proves that subset partition graphs with the dimension reduction property have superlinear diameter, which can be considered evidence against the Linear Hirsch Conjecture.
\end{remark}

The following construction due to Santos (\cite{Santos:PersonalCommunication2011}, which improves the author's previously unpublished bound) gives a superpolynomial lower bound for subset partition graphs with the adjacency and endpoint-count conditions:
\begin{theorem}\label{theorem:general-lower}
Let $K(n,d)$ denote the maximum diameter of $d$-dimensional subset partition graphs with $n$ symbols satisfying the strong adjacency, endpoint-count, and one-subset properties. There is a universal constant $\kappa$ such that 
\[ \frac{K(n,d)}{n^{d/4}} \geq \kappa > 0\]
for infinitely many $n$ and $d$.
\end{theorem}
\begin{proof}
Let $d \geq 8$ be a multiple of four. Let $n > d$ be even. We construct a $d$-dimensional subset partition graph $G = (\mathcal{V},E)$ with $n$ symbols. The symbol set is $S := \{1,\ldots,k\} \cup \{1',\ldots,k'\}$, where $k = \frac{n}{2}$.

Let $P$ be a $\frac{d}{2}$-dimensional cyclic polytope with $k$ vertices. The polar of $P$ is a simple $\frac{d}{2}$-polytope with $k$ facets. Let $Q$ and $Q'$ be two copies of the polar of $P$, with respective symbol sets $\Sigma = \{1,\ldots,k\}$ and $\Sigma' = \{1',\ldots,k'\}$ labeling the facets so that the involution $f : s \mapsto s'$ is a combinatorial bijection. Since $Q$ has a Hamiltonian path $\Pi$ (see~\cite{Klee:PathsII}), we order the $t$ vertices of $Q$ as $Z_1, \ldots, Z_{t}$ using the path $\Pi$. Note 
\[t = \frac{n}{n-\frac{d}{2}}\binom{\frac{n}{2}-\frac{d}{4}}{\frac{d}{4}}
.\]
(Here, each $Z_i$ is a $\frac{d}{2}$-subset of $\Sigma$, thus their images $(Z_i')_{i=1}^{t}$ with $Z_i' = f(Z_i)$ trace the same Hamiltonian path in $Q'$.)

The vertex set $\mathcal{V}$ of $G$ is $\{\mathcal{V}_i : i = 1,\ldots,t\} \cup \{\mathcal{W}_{i,j} : i = 1,\ldots,t-1;\, j =1,2\}$ and each $\mathcal{W}_{i,j}$ has edges to $\mathcal{V}_i$ and $\mathcal{V}_{i+1}$. See Figure~\ref{figure:almost-path-graph}.
\begin{figure}[hbt]
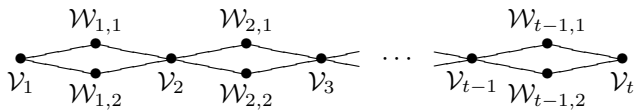

\begin{center}
\[ 
\xy 
(0,0)*+{\bullet};
(20,0)*+{\bullet};
(40,0)*+{\bullet};
(60,0)*+{\bullet};
(80,0)*+{\bullet};
(10,2)*+{\bullet};
(30,2)*+{\bullet};
(70,2)*+{\bullet};
(10,-2)*+{\bullet};
(30,-2)*+{\bullet};
(70,-2)*+{\bullet};
(0,-3.1)*+{\mathcal{V}_1};
(20,-3.1)*+{\mathcal{V}_2};
(40,-3.1)*+{\mathcal{V}_3};
(60,-3.1)*+{\mathcal{V}_{t-1}};
(80,-3.1)*+{\mathcal{V}_t};
(10,5)*+{\mathcal{W}_{1,1}};
(30,5)*+{\mathcal{W}_{2,1}};
(70,5)*+{\mathcal{W}_{t-1,1}};
(10,-5)*+{\mathcal{W}_{1,2}};
(30,-5)*+{\mathcal{W}_{2,2}};
(70,-5)*+{\mathcal{W}_{t-1,2}};
(50,0)*+{\cdots};
(0,0); (10,-2) **\dir{-};
(20,0); (30,-2) **\dir{-};
(40,0); (45,-1) **\dir{-};
(60,0); (70,-2) **\dir{-};
(20,0); (10,-2) **\dir{-};
(40,0); (30,-2) **\dir{-};
(60,0); (55,-1) **\dir{-};
(80,0); (70,-2) **\dir{-};
(0,0); (10,2) **\dir{-};
(20,0); (30,2) **\dir{-};
(40,0); (45,1) **\dir{-};
(60,0); (70,2) **\dir{-};
(20,0); (10,2) **\dir{-};
(40,0); (30,2) **\dir{-};
(60,0); (55,1) **\dir{-};
(80,0); (70,2) **\dir{-};
\endxy 
\]
\end{center}
\caption{The underlying graph of $G = (\mathcal{V},E)$}\label{figure:almost-path-graph}
\end{figure}
Each vertex $\mathcal{V}_i$ and $\mathcal{W}_{i,j}$ contains exactly one $d$-set. Define $\mathcal{V}_i = \{A_i\}$, where $A_i = Z_i \cup Z_i'$. The sets $A_i$ and $A_{i+1}$ have $d-2$ elements in common, so each $D_i := A_i \vartriangle A_{i+1}$ has cardinality $4$, where $\vartriangle$ denotes symmetric difference.

\begin{figure}[hbt]
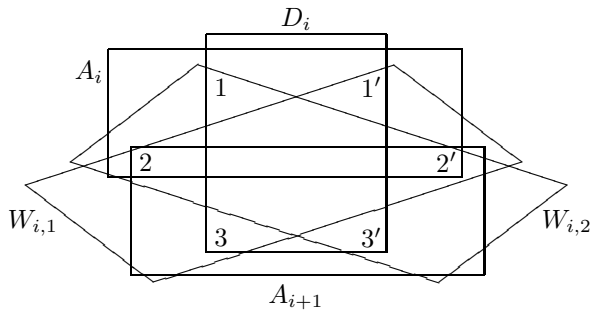

\begin{center}
\[ 
\xy 
(-10,10)*+{1};
(-20,0)*+{2};
(-10,-10)*+{3};
(10,10)*+{1'};
(20,0)*+{2'};
(10,-10)*+{3'};
(-25,15); (22,15) **\dir{-};
(-25,-2); (22,-2) **\dir{-};
(-25,15); (-25,-2) **\dir{-};
(22,15); (22,-2) **\dir{-};%
(-27.5, 12)*+{A_i};
(-22,-15); (25,-15) **\dir{-};
(-22,2); (25,2) **\dir{-};
(-22,-15); (-22,2) **\dir{-};
(25,-15); (25,2) **\dir{-};%
(0, -18)*+{A_{i+1}};
(-12, 17); (12,17); **\dir{-};
(-12, -12); (12,-12); **\dir{-};
(-12, 17); (-12,-12); **\dir{-};
(12, 17); (12,-12); **\dir{-};%
(0, 19)*+{D_i};
(-13,13); (36, -3); **\dir{-};
(-30,0); (19, -16); **\dir{-};
(-13,13); (-30,0); **\dir{-};
(36,-3); (19,-16); **\dir{-};%
(36,-8)*+{W_{i,2}};
(13,13); (-36, -3); **\dir{-};
(30,0); (-19, -16); **\dir{-};
(13,13); (30,0); **\dir{-};
(-36,-3); (-19,-16); **\dir{-};%
(-35,-8)*+{W_{i,1}};
\endxy 
\]
\end{center}
\caption{Venn diagram with $W_{i,\ell}$ for $A_i = \{1,2,1',2'\}$ and $A_{i+1} = \{2,3,2',3'\}$}\label{figure:two-W-sets}
\end{figure}

Consider the two $d$-sets $W_{i,1} = A_i \cup (A_{i+1} \cap D_i \cap \Sigma) \setminus (A_i \cap D_i \cap \Sigma)$ and $W_{i,2} = A_i \cup (A_{i+1} \cap D_i \cap \Sigma') \setminus (A_i \cap D_i \cap \Sigma')$. See Figure~\ref{figure:two-W-sets} for an example. Let $\mathcal{W}_{i,\ell}$ contain $W_{i,\ell}$. It is easy to see that $|A_i \cap A_j| \leq d-2$ if $i \not= j$. It follows that $G$ satisfies the strong adjacency and endpoint-count properties. The diameter of $G$, which is the distance from $\mathcal{V}_1$ to $\mathcal{V}_t$ is $2(t-1)$, which is $\Omega(n^{d/4})$.
\end{proof}
This proves that subset partition graphs without the dimension reduction property do not satisfy the quasi-polynomial bound in~\cite{Kalai:Quasi-polynomial}. Since the diameter of the subset partition graphs in Theorem~\ref{theorem:general-lower} is larger than the best known bound for polytopes, this class may provide a good starting point for the strategy we present to disprove the Linear Hirsch Conjecture which we discuss in Section~\ref{section:final-remarks}.

We now show that a quadratic lower bound holds for a special class of subset partition graphs inspired by spindles, which were instrumental in Santos' disproof of the Hirsch Conjecture~\cite{Santos:CounterexampleHirsch}. A \emph{spindle} is a polytope $P$ with two special vertices $A_1$ and $A_2$ (called the \emph{apices}) such that every facet of $P$ contains exactly one of the apices. The \emph{length} of a spindle $P$ is the distance in the graph of $P$ between $A_1$ and $A_2$. We say that a $d$-dimensional spindle is \emph{long} if its length exceeds $d$. In~\cite{Santos:CounterexampleHirsch}, Santos proved that long $5$-dimensional spindles exist, and moreover, the existence of a long spindle implies the existence of a long spindle with $n=2d$, thus the Hirsch Conjecture for polytopes is false. (Recall that a \emph{Dantzig figure} is an abstract polytope $G= (\mathcal{A},E)$ in the sense of~\cite{AdlerDantzigMurty:AbstractPolytopes} where $n=2d$, together with two disjoint $d$-sets $A_1,A_2 \in \mathcal{A}$. The construction of Santos is a polytopal realization of a non-Hirsch Dantzig figure.)

The property that characterizes spindles is purely combinatorial, so we make an analogous definition for subset partition graphs. We say that a subset partition graph $G$ of $\mathcal{A}$ on the symbol set $S$ \emph{satisfies the spindle property} if there are two distinguished subsets $A_1$ and $A_2$ (called the \emph{apices}) belonging to $\mathcal{A}$, such that every symbol $s \in S$ belongs to exactly one of $A_1$ or $A_2$. The \emph{length} of a subset partition graph $G$ with the spindle property is the distance in $G$ from one apex to the other.
\begin{theorem}\label{theorem:spg-spindle-length}
Let $L(n)$ denote the maximum length of $d$-dimensional subset partition graphs on $n=2d$ symbols satisfying the strong adjacency, endpoint-count, one-subset, and spindle properties.
Let $\kappa = \frac18$. For infinitely many $n$,
\[\frac{L(n)}{n^2} \geq \kappa.\]
\end{theorem}
\begin{proof}
We define a subset partition graph $G_m = (\mathcal{V},E)$ of $\mathcal{A}$ for each $m \in \mathbb{N}$. Let $S = [2m] \times \{1,2\}$, so $n = |S| = 4m$. Define the index set
\[ I := \{(a,b,c) : 0 \leq a, b < m;\, c=0,1\} \cup \{(m,0,0)\}.\]
The triples $(a,b,c) \in I$ are totally ordered using the lexicographic order $<_{\text{lex}}$. Note that since $c \in \{0,1\}$, if $a=a'$ and $b=b'$, then $(a,b,c)$ and $(a',b',c')$ are either equal or consecutive in $<_{\text{lex}}$. For $(a,b,c) \in I$, define
\begin{align*}
A_{a,b,c} \, := 
     \,\, &\{ (i,j) : i = a+1, \ldots, a+m-b-1 \text{ and } j = 1, 2\} \\
\cup \,\, &\{ (i,j) : i = a+m-b \text{ and } j = c+1, \ldots, 2\} \\
\cup \,\, &\{ (i,j) : i = a+m-b+1 \text{ and } j = 1, \ldots, c\} \\
\cup \,\, &\{ (i,j) : i = a+m-b+2,\ldots, a+m+1 \text{ and } j = 1, 2\}.
\end{align*}
Figure~\ref{figure:typical-d-set} gives schematic pictures of typical sets $A_{a,b,c}$ in $\mathcal{A}$.  Let $\mathcal{A} := \{A_{a,b,c} : (a,b,c) \in I\}$. For all $(a,b,c) \in I$, one has $|A_{a,b,c}| = 2m$, so the dimension of $G_m$ is $d=2m=\frac{n}{2}$.
\begin{figure}[hbt]
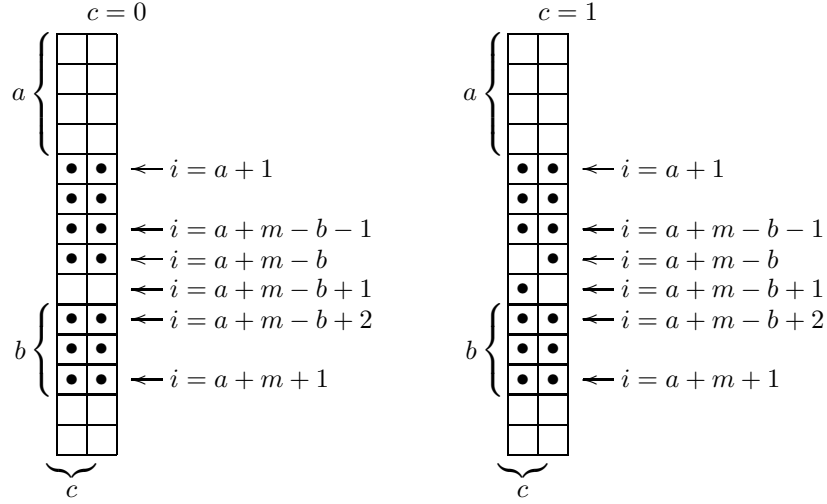

\begin{center}
\[ 
\xy 
(-52,3)*+{c=0};
(-60,0); (-60,-56) **\dir{-};
(-56,0); (-56,-56) **\dir{-};
(-52,0); (-52,-56) **\dir{-};
(-60,0); (-52, 0) **\dir{-};
(-60,-4); (-52,-4) **\dir{-};
(-60,-8); (-52,-8) **\dir{-};
(-60,-12); (-52,-12) **\dir{-};
(-60,-16); (-52,-16) **\dir{-};
(-60,-20); (-52,-20) **\dir{-};
(-60,-24); (-52,-24) **\dir{-};
(-60,-28); (-52,-28) **\dir{-};
(-60,-32); (-52,-32) **\dir{-};
(-60,-36); (-52,-36) **\dir{-};
(-60,-40); (-52,-40) **\dir{-};
(-60,-44); (-52,-44) **\dir{-};
(-60,-48); (-52,-48) **\dir{-};
(-60,-52); (-52,-52) **\dir{-};
(-60,-56); (-52,-56) **\dir{-};
(-58,-18)*+{\bullet};
(-54,-18)*+{\bullet}; 
{\ar (-46,-18)*+!L{i=a+1}; (-50,-18)*{}};
(-58,-22)*+{\bullet}; 
(-54,-22)*+{\bullet}; 
(-58,-26)*+{\bullet}; 
(-54,-26)*+{\bullet}; 
{\ar (-46,-26)*+!L{i=a+m-b-1}; (-50,-26)*{}};
(-54,-30)*+{\bullet}; 
(-58,-30)*+{\bullet};
{\ar (-46,-30)*+!L{i=a+m-b}; (-50,-30)*{}}; 
{\ar (-46,-34)*+!L{i=a+m-b+1}; (-50,-34)*{}}; 
(-58,-38)*+{\bullet}; 
(-54,-38)*+{\bullet}; 
{\ar (-46,-38)*+!L{i=a+m-b+2}; (-50,-38)*{}};
(-58,-42)*+{\bullet}; 
(-54,-42)*+{\bullet}; 
(-58,-46)*+{\bullet}; 
(-54,-46)*+{\bullet}; 
{\ar (-46,-46)*+!L{i=a+m+1}; (-50,-46)*{}};
(-59.3,-58).(-56.8,-58)!C *\frm{_\}},+U*++!U\txt{$c$};
(-62,0).(-62,-16)!C *\frm{\{},+L*++!R\txt{$a$};
(-62,-36).(-62,-48)!C *\frm{\{},+L*++!R\txt{$b$};
(8,3)*+{c=1};
(0,0); (0,-56) **\dir{-};
(4,0); (4,-56) **\dir{-};
(8,0); (8,-56) **\dir{-};
(0,0); (8, 0) **\dir{-};
(0,-4); (8,-4) **\dir{-};
(0,-8); (8,-8) **\dir{-};
(0,-12); (8,-12) **\dir{-};
(0,-16); (8,-16) **\dir{-};
(0,-20); (8,-20) **\dir{-};
(0,-24); (8,-24) **\dir{-};
(0,-28); (8,-28) **\dir{-};
(0,-32); (8,-32) **\dir{-};
(0,-36); (8,-36) **\dir{-};
(0,-40); (8,-40) **\dir{-};
(0,-44); (8,-44) **\dir{-};
(0,-48); (8,-48) **\dir{-};
(0,-52); (8,-52) **\dir{-};
(0,-56); (8,-56) **\dir{-};
(2,-18)*+{\bullet};
(6,-18)*+{\bullet}; 
{\ar (14,-18)*+!L{i=a+1}; (10,-18)*{}};
(2,-22)*+{\bullet}; 
(6,-22)*+{\bullet}; 
(2,-26)*+{\bullet}; 
(6,-26)*+{\bullet}; 
{\ar (14,-26)*+!L{i=a+m-b-1}; (10,-26)*{}};
(6,-30)*+{\bullet}; 
{\ar (14,-30)*+!L{i=a+m-b}; (10,-30)*{}}; 
(2,-34)*+{\bullet};
{\ar (14,-34)*+!L{i=a+m-b+1}; (10,-34)*{}}; 
(2,-38)*+{\bullet}; 
(6,-38)*+{\bullet}; 
{\ar (14,-38)*+!L{i=a+m-b+2}; (10,-38)*{}};
(2,-42)*+{\bullet}; 
(6,-42)*+{\bullet}; 
(2,-46)*+{\bullet}; 
(6,-46)*+{\bullet}; 
{\ar (14,-46)*+!L{i=a+m+1}; (10,-46)*{}};
(0.7,-58).(3.2,-58)!C *\frm{_\}},+U*++!U\txt{$c$};
(-2,0).(-2,-16)!C *\frm{\{},+L*++!R\txt{$a$};
(-2,-36).(-2,-48)!C *\frm{\{},+L*++!R\txt{$b$};
\endxy 
\]
\end{center}
\caption{Schematic drawings of the $d$-sets $A_{a,b,c} \in \mathcal{A}$, $m=7$ for $c=0$ and $c=1$}\label{figure:typical-d-set}
\end{figure}

The elements in $\mathcal{V}$ are the singleton sets $\mathcal{V}_{a,b,c} = \{A_{a,b,c}\}$ for $(a,b,c) \in I$, thus the one-subset property holds. Two vertices $\mathcal{V}_{a,b,c}$ and $\mathcal{V}_{a',b',c'}$ are connected by an edge if and only if the triples $(a,b,c)$ and $(a',b',c')$ are consecutive in $<_{\text{lex}}$, which is a total order on $I$. Therefore the graph for $G_m$ is a path, so the graph is connected. Hence, $G_m$ is a subset partition graph.

The $d$-sets $A_{0,0,0} = \{(i,j) : 1 \leq i \leq m; 1 \leq j \leq 2\}$ and $A_{m,0,0} = \{(i,j) : m+1 \leq i \leq 2m, 1 \leq j \leq 2\}$ are a disjoint partition of $S$, thus $A_{0,0,0}$ and $A_{m,0,0}$ are the apices of the spindle $G_m$.

We prove the adjacency property by showing that if the $d$-sets $A:=A_{a,b,c}$ and $A':=A_{a',b',c'}$ are not in adjacent vertices, then $|A \cap A'| < d-1$, or equivalently, $|A \vartriangle A'| \geq 2$. Since $(a,b,c) \not= (a',b',c')$, without loss of generality $(a,b,c) <_{\text{lex}} (a',b',c')$. For integers $\alpha,\beta,\gamma,\iota_1,\iota_2,\dots$ define 
\[ I_{\alpha,\beta,\gamma}(\iota_1,\iota_2,\dots) := | \{ (i,j) \in A_{\alpha,\beta,\gamma} : i = \iota_1,\iota_2,\dots \} |, \]
and use $I(\iota_1,\iota_2,\dots) := I_{a,b,c}(\iota_1,\iota_2,\dots)$ and $I'(\iota_1,\iota_2,\dots) := I_{a',b',c'}(\iota_1,\iota_2,\dots)$ respectively to denote the number of elements with $i \in \{\iota_1,\iota_2,\dots\}$ in $A$ and $A'$ respectively.  The proof is given in cases.
\begin{itemize}
\item Suppose $a' = a$. Since $A$ and $A'$ are not adjacent, $b' \not= b$. Either $b'-b=1$ or $b'-b>1$.
\begin{itemize}

\item 
Suppose $b'-b=1$. If $c=0$, then $I(a+m-b+1)=0$ but $I'(a+m-b+1)=I'(a'+m-b'+2)=2$. If $c=1$, then non-consecutivity of $(a,b,c)$ and $(a',b',c')$ implies $c'=1$, hence both $(a+m-b-1,1)$ and $(a+m-b,2)$ are in $A$ but neither are in $A'$. In either case $|A \vartriangle A'| \geq 2$.

\item 
If $b'-b>1$, then $I(a+m-b,a+m-b+1) = 2$ but $I'(a+m-b,a+m-b+1) = 4$, so $|A \vartriangle A'| \geq 2$.
\end{itemize}

\item Otherwise, $a' \not= a$. Exactly one of the following holds:
\begin{enumerate}
\item $a' - a > 1$,
\item $a' - a = 1$ and $b < m-1$,
\item $a' - a = 1$ and $b = m-1$ and $c=0$,
\item $a' - a = 1$ and $b = m-1$ and $c=1$ and $b' > 0$, or
\item $a' - a = 1$ and $b = m-1$ and $c=1$ and $b' = 0$.
\end{enumerate}
In the first three cases, $I(a+1)=2$ but $I'(a+1)=0$. In the fourth case, $I'(a'+m+1)=2$ but $I(a'+m+1)=0$. In the last case, non-consecutivity of $(a,b,c)$ and $(a',b',c')$ implies $c'=1$. Then $(a+1,2)$ and $(a+m+1,1)$ are in $A$, but neither are in $A'$. In all cases, $|A \vartriangle A'| \geq 2$.
\end{itemize}
Since non-adjacent $d$-sets $A$ and $A'$ have strictly less than $d-1$ symbols in common, the adjacency property holds. Moreover, it is clear from the definition of $A_{a,b,c}$ that if $A$ and $A'$ are adjacent then $|A \cap A'| = d-1$, so strong adjacency holds. Thus endpoint-count holds since the degree of each vertex in $G_m$ is one or two.

The underlying graph of $G_m$ is a path and its length, from $A_{0,0,0}$ to $A_{m,0,0}$, is $2m^2 = \frac18n^2$, thus $\limsup_{n \rightarrow \infty} \frac{L(n)}{n^{2}} \geq 2^{-3}$.
\end{proof}
Since the length is a lower bound for the diameter:
\begin{corollary}\label{corollary:superlinear-bound}
Let $M(n)$ denote the maximum diameter of $d$-dimensional subset partition graphs on $n=2d$ symbols satisfying the strong adjacency, endpoint-count, and spindle properties. Then,
\[ \limsup_{n \rightarrow \infty} \frac{M(n)}{n^2} \geq \kappa = \frac18.\]
\end{corollary}
Very recently, the lower bound to $M(n)$ was improved by H\"{a}hnle in~\cite{Hahnle:SPGs}.

\section{Final remarks and open problems}\label{section:final-remarks}

We saw that the Kalai-Kleitman diameter upper bound in~\cite{Kalai:Quasi-polynomial} holds for subset partition graphs satisfying dimension reduction. While we do have a lower bound for diameters of subset partition graphs with the strong adjacency and endpoint-count conditions, we ask:
\begin{problem}
Prove a non-trivial upper bound on the diameters of subset partition graphs with the strong adjacency and endpoint-count conditions.
\end{problem}
Subset partition graphs that satisfy the first main property, namely dimension reduction (see Remark~\ref{remark:spg-lower-bound-from-clf}), or the last two main properties, namely strong adjacency and endpoint-count (see Theorem~\ref{theorem:general-lower}), have superlinear diameter. Both of these results, combined with the fact that complementary sets of conditions are used, can be considered evidence against the Linear Hirsch Conjecture. Theorem~\ref{theorem:spg-spindle-length}, which presents a quadratic diameter lower bound for a special subclass provides even further evidence against the Linear Hirsch Conjecture. In light of this, we ask:
\begin{problem}
Construct a family of subset partition graphs with superlinear diameter satisfying all of the main properties.
\end{problem}
In fact, subset partition graphs provide an approach for satisfying all three properties and, moreover, an approach for disproving the Linear Hirsch Conjecture:
\begin{enumerate}
\item Start with a family of subset partition graphs satisfying at least the endpoint-count property with superlinear diameter growth, such as the family resulting from either Theorems~\ref{theorem:general-lower} or~\ref{theorem:spg-spindle-length}.
\item Gain the other main properties that do not yet hold with the contraction and edge addition operations (see Remark~\ref{remark:effect-of-operations}).
\item If the resulting family of graphs still has superlinear diameter, realize the sequence of graphs as a sequence of polytopes.
\end{enumerate}
We identify the principal difficulties with this strategy: in the second step the contraction and edge addition operations are liable to significantly reduce the diameter of the subset partition graphs, and in the third step the realization problem for polytopes and the study of polytopality of graphs is still the subject of ongoing research, e.{}g.{},~\cite{Joswig:NeighborlyCubical}, \cite{Matschke:Prodsimplicial}, \cite{Pfeifle:PolytopalityCartesian}, \cite{Pilaud:MultiPseudoTriangulationsRealization}, \cite{Ziegler:ProjectedProductsPolygons}.

In light of these difficulties, for the purpose of the above approach, we note that \emph{any} superlinear construction of subset partition graphs satisfying the endpoint-count property is useful, since the method attempts to construct polytopes using some construction of subset partition graphs as a starting point: while the bound in Theorem~\ref{theorem:general-lower} is much better than the one in Theorem~\ref{theorem:spg-spindle-length}, it could be that steps 2 and 3 above are easier to perform from the construction in Theorem~\ref{theorem:spg-spindle-length}. Thus, any superlinear construction of subset partition graphs satisfying endpoint-count is interesting.

By considering different combinations of properties, subset partition graphs provide a framework for describing which conditions are crucial in proving upper and lower bounds. It is natural to ask which combination of properties is most useful in combinatorial abstractions for superlinear lower bounds:
\begin{question}
Are there superlinear lower bounds for subset partition graphs satisfying other non-trivial combinations of properties?
\end{question}
Finally, is a certain combination of properties sufficient for proving the Polynomial Hirsch Conjecture for subset partition graphs, and thus, for polyhedra?

\section*{Acknowledgments}
I thank Jes\'us A. De Loera, Nicolai H\"ahnle, Frederik von Heymann, Gil Kalai, Vincent Pilaud, Frank Vallentin, and G\"unter M. Ziegler for their remarks. I am especially grateful to Francisco Santos for the in-depth suggestions. I express my gratitude to the anonymous referees for their helpful comments.

\noindent{\small Edward D. Kim -- {\tt edward.kim.math@gmail.com}\\
Delft Institute of Applied Mathematics\\
Technische Universiteit Delft\\
Mekelweg 4\\
2628 CD  Delft\\
The Netherlands}

\end{document}